\documentclass[11pt]{amsart}
\usepackage[dvipsnames]{xcolor}
\usepackage[margin=1in]{geometry}
\usepackage{amsmath,amssymb,amsthm,mathtools}
\usepackage{graphicx}
\usepackage{tikz}
\usepackage{nicefrac}
\usepackage[square]{natbib}
\usepackage[citecolor=PineGreen,colorlinks=true,linkcolor=RoyalBlue,urlcolor=BrickRed]{hyperref}

\newtheorem{theorem}{Theorem}[section]
\newtheorem{proposition}[theorem]{Proposition}
\newtheorem{lemma}[theorem]{Lemma}
\numberwithin{equation}{section}

\newcommand{\Z}{\mathbb{Z}}
\newcommand{\R}{\mathbb{R}}
\newcommand{\E}{\mathbb{E}}
\newcommand{\PP}{\mathbb{P}}
\newcommand{\ba}{\mathbf{a}}
\newcommand{\cO}{\mathcal{O}}
\newcommand{\eps}{\varepsilon}
\newcommand{\Leps}{\mathcal{L}^\varepsilon}
\newcommand{\Lbar}{\bar{\mathcal{L}}^\varepsilon}
\newcommand{\sgn}{\operatorname{sgn}}
\newcommand{\nf}{\nicefrac}

\title[Homogenization for the critical long-range random conductance model]{Quantitative homogenization for the critical long-range random conductance model}

\author{Ahmed Bou-Rabee}
\address{University of Pennsylvania, United States of America}
\email{ahmedmb@sas.upenn.edu}

\author{Paul Dario}
\address{CY Cergy Paris Universit\'e, France}
\email{paul.dario@cyu.fr}


\begin{document}

\begin{abstract}
We consider the long-range random conductance model on $\Z^d$ at the critical exponent: the jump rate between sites $x$ and $y$ is $\ba(\{x,y\}) |x-y|^{-(d+2)}$, where the conductances $\ba(\{x,y\})$ are independent, uniformly elliptic and have a common mean.
Below the critical exponent~$(d+2)$ the walk converges to a stable process; above it, to Brownian motion with diffusive $\sqrt{t}$ scaling.
At criticality the second moment of the jump kernel diverges logarithmically.
We establish quantitative homogenization of the associated heat equation at the rate $1/\sqrt{|\ln \eps|}$.
As a consequence, we deduce annealed convergence of the random walk to Brownian motion under the anomalous $\sqrt{t \log t}$ scaling.
Unlike in standard homogenization, the effective diffusivity is determined by the mean conductance alone, with no corrector contribution at leading order.
\end{abstract}

\maketitle

\begin{figure}[h!]
\centering
\includegraphics[width=\textwidth]{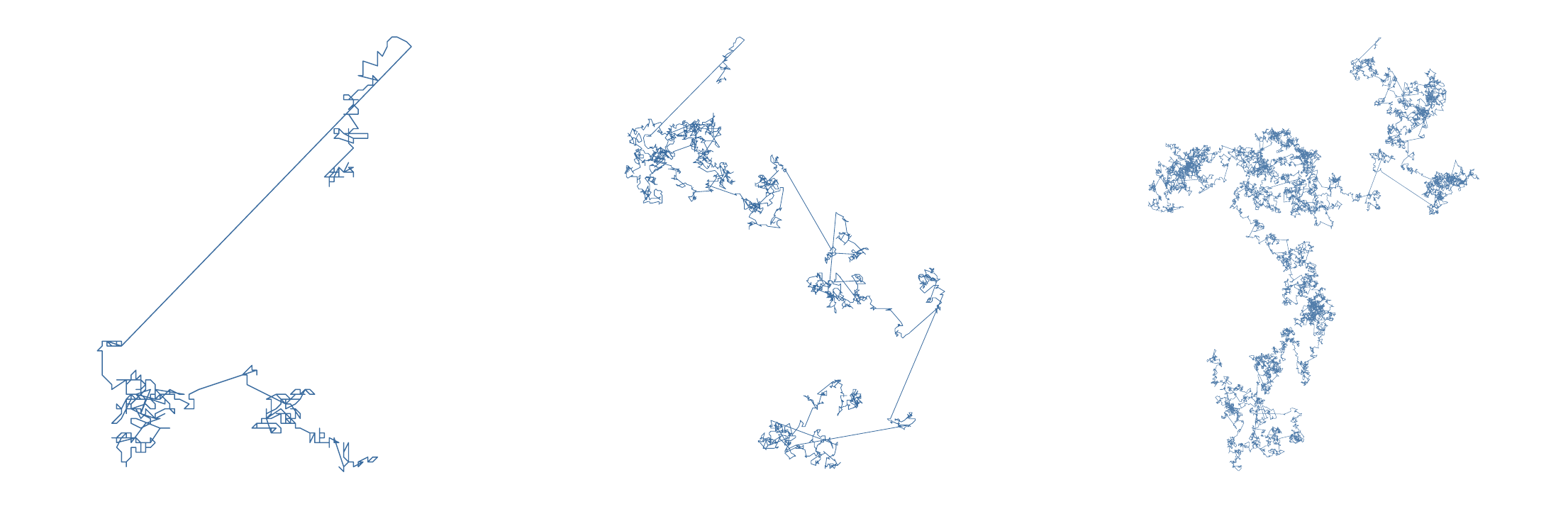}
\caption{A sample path of the critical long-range random walk on $\Z^2$ with jump kernel $|z|^{-4}$ and random conductances, shown at three successive scales: $500$ steps (left), $5{,}000$ steps (center), and $100{,}000$ steps (right). By Theorem~\ref{thm:IP}, the walk converges to Brownian motion under $\sqrt{t \log t}$ scaling.}
\label{fig:qip}
\end{figure}

\section{Introduction}

\subsection{The critical long-range random conductance model}

Consider a random walk on $\Z^d$ whose jump rate from $x$ to $y$ is
\begin{equation} \label{eq:jump.rates}
    \ba(\{x,y\}) |x - y|^{-d-\alpha}\,,
\end{equation}
where $\alpha > 0$, the norm $|\cdot|$ is Euclidean, and the conductances $\ba(\{x,y\})$ are random and take values in $[\lambda, \lambda^{-1}]$ for a fixed $\lambda \in (0,1]$.
The exponent $\alpha$ controls the large-scale behavior.
For $\alpha < 2$, the rescaled walk converges to an $\alpha$-stable process~\citep*{CKW2021, CCKW2025}.
For $\alpha > 2$, it converges to Brownian motion~\citep*{BCKW2021}.
At the critical value $\alpha = 2$, the second moment $\sum_{z \in \Z^d \setminus \{0\}} |z|^2 \cdot |z|^{-d-2}$ of the jump rates diverges logarithmically, and the walk is superdiffusive by a logarithmic factor (see Figure~\ref{fig:qip}).

This paper considers the critical case, with independent conductances: we prove quantitative homogenization of the associated heat equation at the rate $|\ln \eps|^{-\nf12}$ (Theorem~\ref{thm:main}), where $\eps$ is the lattice spacing, and deduce an annealed invariance principle under the scaling $\sqrt{t \ln t}$ (Theorem~\ref{thm:IP}).

\subsection{Main results}

Let $d \geq 1$ be an integer and let $\lambda \in (0,1]$.
Let $\ba = (\ba(\{x,y\}))$ be independent random variables, indexed by the unordered pairs $\{x,y\}$ of distinct vertices of $\Z^d$, with values in $[\lambda, \lambda^{-1}]$ and with mean $1$.
We denote by $\PP$ their law and by $\E$ the corresponding expectation.
The conductances need not be identically distributed.
The value $1$ of their common mean is a normalization, removed by rescaling time.

Let $\eps \in (0, \nf12)$.
For $x, z \in \eps \Z^d$ with $z \neq 0$, let $a_{x,z} := \ba(\{x/\eps, (x+z)/\eps\})$ be the conductance of the jump from $x$ to $x + z$.
We define the normalizing constant
\begin{equation} \label{eq:def.kappa}
    \kappa_\eps := \eps^d \sum_{\substack{z \in \eps \Z^d \\ 0 < |z| \leq 1}} \frac{1}{|z|^d}
\end{equation}
and the operator $\Leps$, acting on bounded functions $v \colon \eps \Z^d \to \R$, by
\begin{equation} \label{eq:def.Leps}
    \Leps v(x) := \frac{\eps^d}{\kappa_\eps}
    \sum_{z \in \eps \Z^d \setminus \{0\}}
    \frac{a_{x,z}}{|z|^{d+2}}
    \bigl( v(x + z) - v(x) \bigr) \qquad \text{for } x \in \eps \Z^d\,.
\end{equation}
The constant $\kappa_\eps$ is the second moment of the jumps of length at most $1$, and a comparison with an integral gives $\kappa_\eps = |\mathbb{S}^{d-1}| \cdot |\ln \eps| + O(1)$, where $|\mathbb{S}^{d-1}|$ is the surface area of the unit sphere.

Let $f \in C^\infty_c(\R^d)$.
We compare the solution $u_\eps \colon [0, \infty) \times \eps \Z^d \to \R$ of the lattice heat equation
\begin{equation} \label{eq:heat.lattice}
    \partial_t u_\eps - \Leps u_\eps = 0 \quad \text{in } (0, \infty) \times \eps \Z^d\,, \qquad u_\eps(0, \cdot) = f \quad \text{on } \eps \Z^d\,,
\end{equation}
with the solution $\bar u \colon [0, \infty) \times \R^d \to \R$ of the continuum heat equation
\begin{equation} \label{eq:heat.continuum}
    \partial_t \bar u - \frac{1}{2d} \Delta \bar u = 0 \quad \text{in } (0, \infty) \times \R^d\,, \qquad \bar u(0, \cdot) = f \quad \text{on } \R^d\,.
\end{equation}
For a function $v \colon \eps \Z^d \to \R$, we use the rescaled norm
\begin{equation} \label{eq:def.L2}
    \| v \|_{L^2(\eps \Z^d)} := \biggl( \eps^d \sum_{x \in \eps \Z^d} |v(x)|^2 \biggr)^{\nf12}\,.
\end{equation}
When we compare a function on $\R^d$ with a function on $\eps \Z^d$, we restrict the function on $\R^d$ to $\eps \Z^d$.
For fixed $\eps$, the operator $\Leps$ is bounded on the space $L^2(\eps \Z^d)$ of functions with finite norm~\eqref{eq:def.L2}, so $u_\eps(t, \cdot) = e^{t \Leps} f$.
For a random variable $X$ and a constant $A > 0$, we write $X \leq \cO_2(A)$ to mean that
\begin{equation} \label{eq:def.O2}
    \E \left[ \exp \left( \left( \frac{\max(X, 0)}{A} \right)^2 \right) \right] \leq 2\,.
\end{equation}
By Markov's inequality, the bound~\eqref{eq:def.O2} implies $\PP[X \geq rA] \leq 2 e^{-r^2}$ for every $r \geq 0$.

\begin{theorem}[Quantitative homogenization of the heat equation] \label{thm:main}
Let $f \in C^\infty_c(\R^d)$, and let $u_\eps$ and $\bar u$ be the solutions of~\eqref{eq:heat.lattice} and~\eqref{eq:heat.continuum}.
Then there exists a constant $C = C(d, \lambda, f) < \infty$ such that, for every $T \geq 0$ and every $\eps \in (0, \nf12)$,
\begin{equation} \label{eq:main}
    \sup_{0 \leq t \leq T} \left\| u_\eps(t, \cdot) - \bar u(t, \cdot) \right\|_{L^2(\eps \Z^d)} \leq \cO_2 \left( \frac{C (1 + T)}{|\ln \eps|^{\nf12}} \right)\,.
\end{equation}
\end{theorem}

The independence of the conductances is used only once, in Lemma~\ref{lem:G}, through Hoeffding's inequality.
We expect that a duality argument improves the rate to $|\ln \eps|^{-1}$; we do not pursue this here.

We also note a deterministic scaling obstruction to a rate better than $|\ln \eps|^{-1}$.
For $g \in C^\infty_c(\R^d)$, let $u^g_\eps$ and $\bar u^g$ denote the solutions of~\eqref{eq:heat.lattice} and~\eqref{eq:heat.continuum} with initial condition $g$.
Let $f \in C^\infty_c(\R^d)$ be not identically zero, and let $f_2 := f(2 \cdot)$.
A change of variables shows that, for every realization of the conductances, every $\eps \in (0, \nf12)$, every $t \geq 0$ and every $x \in \eps \Z^d$,
\begin{equation} \label{eq:scaling}
    u^{f_2}_{\eps/2} \Bigl( t, \frac{x}{2} \Bigr) = u^f_\eps \Bigl( \frac{4 \kappa_\eps}{\kappa_{\eps/2}} t, x \Bigr)
    \quad \text{and} \quad
    \bar u^{f_2} \Bigl( t, \frac{x}{2} \Bigr) = \bar u^f (4t, x)\,.
\end{equation}
The difference $\kappa_{\eps/2} - \kappa_\eps = \sum_{w \in \Z^d, 1/\eps < |w| \leq 2/\eps} |w|^{-d}$ converges to $|\mathbb{S}^{d-1}| \ln 2$ as $\eps \to 0$, so the two time arguments $4 \kappa_\eps t / \kappa_{\eps/2}$ and $4t$ on the right sides of~\eqref{eq:scaling} differ by the factor $1 - (\ln 2 + o(1)) |\ln \eps|^{-1}$.
For every $t > 0$, the function $\partial_t \bar u^f(4t, \cdot)$ is not identically zero.
Hence, for every realization of the conductances and every $T > 0$, the left sides of~\eqref{eq:main} for $f$ at spacing $\eps$ and for $f_2$ at spacing $\eps/2$ cannot both be $o(|\ln \eps|^{-1})$ as $\eps \to 0$.

Independently of the present work, \citet*[Theorem~1.1]{CGW2026} obtained a quenched rate for the resolvent equation, under the same assumptions on the conductances and in dimensions $d > 3$.
They proved that, for every $\gamma > 0$ and every right-hand side in a dense class of functions, the difference between the resolvent of the walk rescaled by a factor $k$ and the resolvent of the limiting Brownian motion is, almost surely, at most $C (\ln k)^{-1/2 + (1+\gamma)/(2(d-2))}$ in $L^2(\R^d)$ for every sufficiently large $k$.

For a realization of the conductances, let $X^\eps = (X^\eps_t)_{t \geq 0}$ be the continuous-time random walk on $\eps \Z^d$ with generator $\Leps$.
This walk jumps from $x$ to $x + z$ at rate $\kappa_\eps^{-1} \eps^d a_{x,z} |z|^{-d-2}$.
We denote by $\PP^\ba_x$ and $\E^\ba_x$ the quenched law and expectation of this walk started from $x \in \eps \Z^d$, and by $\PP^{\mathrm{ann}}_0 := \E[\PP^\ba_0]$ and $\E^{\mathrm{ann}}_0$ the annealed law and expectation.
Let $D([0, \infty), \R^d)$ be the space of c\`adl\`ag functions from $[0,\infty)$ to $\R^d$, equipped with the Skorokhod topology.

\begin{theorem}[Annealed invariance principle] \label{thm:IP}
Let $X^\eps$ be the random walk on $\eps \Z^d$ with generator~\eqref{eq:def.Leps} started from $0$, and let $W$ be a standard Brownian motion in $\R^d$.
Then, as $\eps \to 0$, the law of $X^\eps$ under $\PP^{\mathrm{ann}}_0$ converges to the law of $d^{-\nf12} W$ in $D([0, \infty), \R^d)$.
\end{theorem}

In the original scaling, Theorem~\ref{thm:IP} states the following.
Let $Y$ be the random walk on $\Z^d$ with jump rates $\ba(\{x,y\}) |x-y|^{-d-2}$ started from $0$.
Then, under the annealed law,
\begin{equation*}
    \bigl( (n \ln n)^{-\nf12} Y_{nt} \bigr)_{t \geq 0} \quad \text{converges in law to} \quad \Bigl( \frac{|\mathbb{S}^{d-1}|}{2d} \Bigr)^{\nf12} W \qquad \text{as } n \to \infty\,.
\end{equation*}
Indeed, the process $(\eps Y_{t/(\kappa_\eps \eps^2)})_{t \geq 0}$ has the law of $X^\eps$.
For the choice $\eps := (\nf12 |\mathbb{S}^{d-1}| n \ln n)^{-\nf12}$, we have $\kappa_\eps \eps^2 n \to 1$ as $n \to \infty$.

The first version of this work~\citep{BD2026} deduced a quenched invariance principle, for walks started from a compactly supported density, from the homogenization of the resolvent equation on bounded domains.
Theorem~\ref{thm:IP} is annealed, and the walk starts from a single vertex.

The limit $d^{-\nf12} W$ has generator $\nf{1}{2d} \Delta$.
If the conductances have common mean $m$ instead of $1$, then the walks with generator~\eqref{eq:def.Leps} converge to a Brownian motion with generator $\nf{m}{2d} \Delta$: the effective diffusivity depends on the law of the conductances only through their mean, with no corrector contribution at leading order.
This contrasts with the nearest-neighbor random conductance model, in which the effective diffusivity is a nontrivial function of the law of the conductances, determined by a corrector equation on the whole lattice~\citep{Kozlov1979, PapanicolaouVaradhan1981}.

\subsection{Proof strategy}

The proof of Theorem~\ref{thm:main} compares the lattice equation~\eqref{eq:heat.lattice} directly with its homogenized limit, without correctors.
Let $\Lbar$ denote the operator~\eqref{eq:def.Leps} with every conductance replaced by its mean $1$.
Inserted into the lattice equation~\eqref{eq:heat.lattice}, the continuum solution $\bar u$ leaves the remainder
\begin{equation} \label{eq:splitting}
    \partial_t \bar u - \Leps \bar u = \bigl( \partial_t \bar u - \Lbar \bar u \bigr) + \bigl( \Lbar - \Leps \bigr) \bar u\,.
\end{equation}
The first term on the right side is deterministic, and the second is random.
The proof requires three ingredients.

The first is a bound on the deterministic term (Lemma~\ref{lem:mean}): a Taylor expansion of $\bar u$ shows that this term is at most $C \kappa_\eps^{-1}$ in $L^2(\eps \Z^d)$.

The second ingredient is a representation of the random term (Proposition~\ref{prop:forcing}).
Tested against a function $v$ on the lattice, the random term is a sum, over the unordered pairs $\{x, x+z\}$, of the independent centered random variables $a_{x,z} - 1$, each multiplied by the product of the increments of $\bar u$ and of $v$ across the pair.
We route each jump of length at most $1$ along a nearest-neighbor path that changes one coordinate at a time (Figure~\ref{fig:path}), which writes the increment of $v$ as a sum of nearest-neighbor increments.
The contribution of these jumps then equals $\kappa_\eps^{-1} \langle G, D v \rangle$, where $D v$ is the discrete gradient of $v$ and $G$ is a centered random vector field.
The jumps of length greater than $1$ contribute at most $C \kappa_\eps^{-1} \|v\|_{L^2(\eps \Z^d)}$.
By Hoeffding's inequality, the norm of $G$ in $L^2(\eps \Z^d)$ is at most $\cO_2(C)$, uniformly in $\eps$ (Lemma~\ref{lem:G}).

The third ingredient is an energy estimate (Section~\ref{sec:energy}).
The difference $u_\eps - \bar u$ solves the lattice heat equation with the two terms on the right side of~\eqref{eq:splitting} as a source and with zero initial condition, and we test this equation against $u_\eps - \bar u$ itself.
The Dirichlet form of $\Leps$ controls $\|D(u_\eps - \bar u)\|_{L^2(\eps \Z^d)}^2$ only up to the factor $\kappa_\eps^{-1}$, so absorbing the term $\kappa_\eps^{-1} \langle G, D(u_\eps - \bar u) \rangle$ costs a square root, and this yields the rate $|\ln \eps|^{-\nf12}$.

Theorem~\ref{thm:IP} follows from Theorem~\ref{thm:main} by a standard route: Aldous' criterion~\citep{Aldous1978} and the convergence of the finite-dimensional marginals.
We verify Aldous' criterion with an off-diagonal heat-kernel bound derived from the stability theory of \citet*{CKW2021Stability} and \citet{Malmquist2023}.
For the marginals, the Markov property reduces the question to the convergence of the semigroup $e^{t \Leps}$, which Theorem~\ref{thm:main} provides in $L^2(\eps \Z^d)$.
Since the walk starts from the single vertex $0$, we use the on-diagonal heat-kernel bound of \citet[Theorem~1.1]{MSC2015} to pass from convergence in $L^2(\eps \Z^d)$ to convergence at the vertex $0$.

\subsection{Context and related work}

For uniformly elliptic divergence-form operators with stationary ergodic coefficients, qualitative homogenization was established by \citet{Kozlov1979}, \citet{PapanicolaouVaradhan1981} and \citet{Yurinskii1980}.
A quantitative theory of stochastic homogenization was initiated by \citet{GO1,GO2}, \citet*{GloriaNeukammOtto2015} and \citet{ArmstrongSmart2016}; see the notes of \citet{ArmstrongKuusi2022} and the book of \citet*{AKM2019}.
These works treat local second-order operators.

The nearest-neighbor random conductance model on $\Z^d$ is surveyed by \citet{Biskup2011}; quenched invariance principles under general moment conditions were obtained by \citet*{AndresDeuschelSlowik2015}.
For long-range models with jump rates comparable to $|x-y|^{-d-\alpha}$ and $\alpha \in (0,2)$, \citet{BassLevin2002} obtained sharp two-sided transition probability estimates on $\Z^d$, \citet{ChenKumagai2008} established heat-kernel estimates for symmetric jump processes, and \citet*{ChenKumagaiWang2020, CKW2021} obtained quenched heat-kernel estimates with possibly degenerate weights and proved a quenched invariance principle under moment conditions on the conductances in dimensions $d > 4 - 2\alpha$.
Qualitative homogenization for the long-range model with stationary ergodic conductances was obtained by \citet*{FlegelHeidaSlowik2019} under a finite second-moment condition on the jump rates, together with lower-tail moment conditions on the nearest-neighbor conductances.
For stationary ergodic conductances, \citet*{BCKW2021} proved a quenched invariance principle in dimensions $d \geq 2$ when the second moment of the jump rates at a vertex has a finite $p$-th moment and the inverses of the nearest-neighbor conductances have a finite $q$-th moment, with $p^{-1} + q^{-1} < 2/d$.
For the jump rates~\eqref{eq:jump.rates}, the second-moment conditions of these two works require $\alpha > 2$.
For $\alpha \in (0,2)$, \citet*{CCKW2025} established quantitative homogenization in dimensions $d > \alpha$, with polynomial rates up to logarithmic factors.

Periodic homogenization for nonlocal integro-differential equations was initiated by \citet{Schwab2010} and extended to convolution-type operators by \citet{PiatnitskiZhizhina2017}.
\citet*{KassmannPiatnitskiZhizhina2019} treated L\'evy-type operators with oscillating coefficients.
On the stochastic side, qualitative homogenization for symmetric stable-like processes in stationary ergodic media was proved by \citet*{ChenChenKumagaiWang2021SIAM}.
None of these works treats the critical exponent.

\citet*{CCKW2021} considered random walks with long-range jumps in \textbf{balanced} random environments, in which the jump rates from $x$ to $x+z$ and from $x$ to $x-z$ are equal for every $x \in \Z^d$ and every $z \in \Z^d \setminus \{0\}$.
They proved a quenched functional central limit theorem for $\alpha \in (0,2]$, with convergence to Brownian motion under the scaling $\sqrt{n \ln n}$ at $\alpha = 2$.
Our walk is reversible and generally unbalanced: the conductances $\ba(\{x,x+z\})$ and $\ba(\{x,x-z\})$ are independent and need not be equal.
Their result and ours therefore address different models.

The scaling $\sqrt{t \ln t}$ reflects that each dyadic scale of jump lengths contributes equally to the effective diffusivity.
Anomalous diffusive scaling arises in polymer models, random media and turbulent transport; see the review of \citet{BouchaudGeorges1990}.
In their classification, the jump rates $|z|^{-d-2}$ have tail index $2$, the boundary between the stable and Gaussian regimes, at which the typical displacement at time $t$ is of order $\sqrt{t \ln t}$.
The same normalization appears in the central limit theorem for the planar periodic Lorentz gas with infinite horizon proved by \citet{SzaszVarju2007}.
More recently, \citet*{ArmstrongBouRabeeKuusi2024} proved a quenched superdiffusive central limit theorem for a Brownian particle in a critically correlated incompressible random drift, and \citet*{CannizzaroMoulardToninelli2025} proved a superdiffusive central limit theorem for the stochastic Burgers equation at the critical dimension.

\subsection{Outline}

Section~\ref{sec:notation} introduces the notation and collects preliminary results.
Section~\ref{sec:proof} proves Theorem~\ref{thm:main}: Section~\ref{sec:mean} bounds the deterministic term, Section~\ref{sec:forcing} treats the random term, and Section~\ref{sec:energy} contains the energy estimate.
Section~\ref{sec:IP} deduces Theorem~\ref{thm:IP} from Theorem~\ref{thm:main}, Aldous' criterion and heat-kernel bounds.

\subsection{Formalization}\label{ssec:formalization}

\begin{sloppypar}
Every labeled theorem, lemma and proposition of this paper has been formalized in Lean~4 and proved from Mathlib and the library \texttt{Lattice-Probability}.
No proof uses \texttt{sorry}, and no proof uses an axiom beyond the three standard axioms \texttt{propext}, \texttt{Classical.choice} and \texttt{Quot.sound}.
The development is available at \url{https://github.com/nitromannitol/critical-long-range-conductance}.
\end{sloppypar}

\subsection*{Use of artificial intelligence}

The first version of this work~\citep{BD2026} was written mostly by the authors, with some help from Claude (Anthropic), and was checked by GPT-5.4 agents (OpenAI).
For the present version, we used GPT-6 Astra (OpenAI) to look for simpler proofs, which led us to work directly with the heat equation, without correctors.
The authors wrote the present paper with assistance from Claude Opus~5.5 (Anthropic) on the exposition, the formatting and the checking of the manuscript.
Claude Fable~5.2 (Anthropic), GPT-6.1 Sol and GPT-6 Astra (OpenAI) audited the manuscript, and Claude Fable~5.2 located the result of \citet{Malmquist2023} used in Step~2 of the proof of Proposition~\ref{prop:hk}.
Claude Opus~5.5 also wrote the Lean formalization, under the direction of the authors.
The authors take responsibility for the mathematical content and presentation of the paper.

\subsection*{Acknowledgments}

During the preparation of the first version of this work, we learned of the independent and simultaneous work of \citet*{CGW2026}.
We thank them for pleasant interactions and for sharing the details of their project.

\section{Notation and preliminary results} \label{sec:notation}

\subsection{Notation} \label{sec:not}

Throughout this article, the symbols $c$ and $C$ denote positive constants that may vary from line to line; the constant $C$ may increase and the constant $c$ may decrease.
Unless otherwise indicated, these constants depend only on the dimension $d$, the ellipticity constant $\lambda$ and the initial condition $f$.
They never depend on $\eps$ or on the conductances.

For $z = (z_1, \ldots, z_d) \in \R^d$, we write $|z|_1 := \sum_{i=1}^d |z_i|$.
Let $e_1, \ldots, e_d$ be the canonical basis of $\R^d$.
We write $J(z) := |z|^{-d-2}$ for $z \in \eps \Z^d \setminus \{0\}$.
A \textbf{short jump} is a vector $z \in \eps \Z^d$ with $0 < |z| \leq 1$, and a \textbf{long jump} is a vector $z \in \eps \Z^d$ with $|z| > 1$.
Both $a_{x,z}$ and $a_{x+z,-z}$ are the conductance of the same unordered pair, so $a_{x,z} = a_{x + z, -z}$.
For two functions $v_1, v_2 \colon \eps \Z^d \to \R$, we write
\begin{equation*}
    \langle v_1, v_2 \rangle := \eps^d \sum_{x \in \eps \Z^d} v_1(x) v_2(x)\,.
\end{equation*}
For $1 \leq i \leq d$, the discrete derivative of a function $v \colon \eps \Z^d \to \R$ in the direction $e_i$ and the norm of its discrete gradient are
\begin{equation*}
    D_i v(x) := \frac{v(x + \eps e_i) - v(x)}{\eps}
    \quad \text{and} \quad
    \|D v\|_{L^2(\eps \Z^d)}^2 := \sum_{i=1}^d \|D_i v\|_{L^2(\eps \Z^d)}^2\,.
\end{equation*}
For a vector field $G = (G_1, \ldots, G_d)$ on $\eps \Z^d$, we write
\begin{equation*}
    \langle G, D v \rangle := \sum_{i=1}^d \langle G_i, D_i v \rangle
    \quad \text{and} \quad
    \|G\|_{L^2(\eps \Z^d)}^2 := \sum_{i=1}^d \|G_i\|_{L^2(\eps \Z^d)}^2\,.
\end{equation*}

\subsection{Preliminary results} \label{sec:prelim}

We first compute the second moments of the short jumps, which give the effective diffusivity $\nf{1}{2d}$ in~\eqref{eq:heat.continuum}.
For every pair of integers $1 \leq i, j \leq d$,
\begin{equation} \label{eq:second.moment}
    \eps^d \sum_{\substack{z \in \eps \Z^d \\ 0 < |z| \leq 1}} J(z) z_i z_j = \frac{\kappa_\eps}{d} \delta_{ij}\,,
\end{equation}
where $\delta_{ij}$ is the Kronecker delta.
Indeed, if $i \neq j$, then the reflection $z_i \mapsto -z_i$ preserves the domain of summation and changes the sign of each term, so the sum vanishes.
The permutations of the coordinates show that the $d$ diagonal sums are equal, and they add up to $\kappa_\eps$ by~\eqref{eq:def.kappa}.

For a realization of the conductances and $v_1, v_2 \in L^2(\eps \Z^d)$, the \textbf{Dirichlet form} of $\Leps$ is the symmetric bilinear form
\begin{equation} \label{eq:def.B}
    B_\ba(v_1, v_2) := \frac{\eps^{2d}}{2 \kappa_\eps} \sum_{x \in \eps \Z^d} \sum_{z \in \eps \Z^d \setminus \{0\}} a_{x,z} J(z) \bigl( v_1(x + z) - v_1(x) \bigr) \bigl( v_2(x + z) - v_2(x) \bigr)\,.
\end{equation}
The next lemma records an integration by parts and a lower bound by $\|D v\|_{L^2(\eps \Z^d)}^2$, which loses the factor $\kappa_\eps$ because it keeps only the nearest-neighbor jumps.

\begin{lemma}[Properties of the Dirichlet form] \label{lem:Dirichlet}
Let $\eps \in (0, \nf12)$.
Then, for every realization of the conductances, the following statements hold.
\begin{enumerate}
    \item[\textup{(i)}] \underline{\emph{Integration by parts}}: for all functions $v_1, v_2 \in L^2(\eps \Z^d)$,
    \begin{equation} \label{eq:IBP}
        \langle \Leps v_1, v_2 \rangle = - B_\ba(v_1, v_2)\,.
    \end{equation}
    \item[\textup{(ii)}] \underline{\emph{Coercivity}}: for every function $v \in L^2(\eps \Z^d)$,
    \begin{equation} \label{eq:coercivity}
        B_\ba(v, v) \geq \frac{c}{\kappa_\eps} \|D v\|_{L^2(\eps \Z^d)}^2\,.
    \end{equation}
\end{enumerate}
\end{lemma}

\begin{proof}
The change of variables $(x, z) \mapsto (x + z, -z)$, together with $a_{x + z, -z} = a_{x,z}$ and $J(-z) = J(z)$, shows that the double sum defining $\langle \Leps v_1, v_2 \rangle$ does not change when $v_2(x)$ is replaced by $-v_2(x+z)$; averaging the two expressions gives~\eqref{eq:IBP}.
For~\eqref{eq:coercivity}, we keep only the vectors $z \in \{\pm \eps e_1, \ldots, \pm \eps e_d\}$ in~\eqref{eq:def.B}, since all the other terms are nonnegative.
These vectors satisfy $J(z) = \eps^{-d-2}$, the conductances satisfy $a_{x,z} \geq \lambda$, and each nearest-neighbor edge appears twice, which cancels the factor $\nf12$ in~\eqref{eq:def.B}.
\end{proof}

We now turn to the continuum solution $\bar u$.
Since we compare lattice functions in the norm~\eqref{eq:def.L2}, we need bounds on the derivatives of $\bar u$ in this norm, uniformly over the translations of the lattice.
For $s \geq 0$, we write $\bar u_s$ for the function $\bar u(s, \cdot)$ and for its restriction to $\eps \Z^d$.

\begin{lemma}[Regularity of the continuum solution] \label{lem:sampling}
Let $f \in C^\infty_c(\R^d)$, let $\bar u$ solve~\eqref{eq:heat.continuum}, and let $\beta$ be a spatial multi-index.
Then there exists a constant $C = C(d, f, \beta) < \infty$ such that, for every $\eps \in (0, \nf12)$, every $s \geq 0$ and every $y \in \R^d$,
\begin{equation} \label{eq:sampling}
    \eps^d \sum_{x \in \eps \Z^d} |\partial^\beta \bar u(s, x + y)|^2 \leq C\,.
\end{equation}
\end{lemma}

\begin{proof}
Let $R \geq 1$ be such that $f$ vanishes outside the ball $B(0,R)$ of radius $R$ centered at $0$.
A ball of radius $R$ contains at most $C R^d \eps^{-d}$ points of a translate of $\eps \Z^d$, so every bounded function $g$ vanishing outside $B(0,R)$ satisfies, for every $y \in \R^d$,
\begin{equation*}
    \eps^d \sum_{x \in \eps \Z^d} |g(x + y)|^2 \leq C \|g\|_{L^\infty(\R^d)}^2\,.
\end{equation*}
Applied to $g = \partial^\beta f$, this proves~\eqref{eq:sampling} for $s = 0$.
For $s > 0$, we have $\partial^\beta \bar u(s, \cdot) = \Phi_s * \partial^\beta f$, where $\Phi_s$ is the heat kernel of~\eqref{eq:heat.continuum}, which is a probability density.
Jensen's inequality gives $|\Phi_s * g(x + y)|^2 \leq \int_{\R^d} \Phi_s(w) |g(x + y - w)|^2 \,dw$.
Summing over $x \in \eps \Z^d$ and applying the bound for $s = 0$ at the points $y - w$, we obtain~\eqref{eq:sampling}.
\end{proof}

Taylor's formula with integral remainder, Minkowski's inequality and Lemma~\ref{lem:sampling} give, for every $s \geq 0$ and every $z \in \R^d$,
\begin{align}
    \|\bar u_s(\cdot + z) - \bar u_s\|_{L^2(\eps \Z^d)} &\leq C \min(|z|, 1)\,, \label{eq:increment} \\
    \Bigl\| \bar u_s(\cdot + z) - \bar u_s - z \cdot \nabla \bar u_s - \tfrac12 z \cdot \nabla^2 \bar u_s z \Bigr\|_{L^2(\eps \Z^d)} &\leq C |z|^3\,. \label{eq:taylor}
\end{align}

Finally, we collect the properties of the notation~\eqref{eq:def.O2} and of the following variant: for a random variable $X$ and a constant $A > 0$, we write $X \leq \cO_1(A)$ to mean that
\begin{equation*}
    \E \left[ \exp \left( \frac{\max(X, 0)}{A} \right) \right] \leq 2\,.
\end{equation*}

\begin{lemma}[Properties of the notation $\cO_q$] \label{lem:O}
The following statements hold.
\begin{enumerate}
    \item[\textup{(i)}] \underline{\emph{Squares}}: for every nonnegative random variable $X$ and every $A > 0$, we have $X \leq \cO_2(A)$ if and only if $X^2 \leq \cO_1(A^2)$.
    \item[\textup{(ii)}] \underline{\emph{Expectation}}: for every random variable $X$ and every $A > 0$ with $X \leq \cO_2(A)$, we have $\E[\max(X, 0)] \leq 2A$.
    \item[\textup{(iii)}] \underline{\emph{Sums}}: let $q \in \{1,2\}$, let $(X_k)_{k \in \mathbb{N}}$ be nonnegative random variables, and let $(A_k)_{k \in \mathbb{N}}$ be positive numbers with $\sum_{k \in \mathbb{N}} A_k < \infty$.
    If $X_k \leq \cO_q(A_k)$ for every $k \in \mathbb{N}$, then
    \begin{equation*}
        \sum_{k \in \mathbb{N}} X_k \leq \cO_q \biggl( \sum_{k \in \mathbb{N}} A_k \biggr)\,.
    \end{equation*}
    \item[\textup{(iv)}] \underline{\emph{Integrals}}: let $q \in \{1,2\}$, let $T > 0$ and $A > 0$, and let $(X(s))_{s \in [0,T]}$ be a jointly measurable family of nonnegative random variables.
    If $X(s) \leq \cO_q(A)$ for every $s \in [0,T]$, then
    \begin{equation*}
        \int_0^T X(s) \,ds \leq \cO_q(T A)\,.
    \end{equation*}
    \item[\textup{(v)}] \underline{\emph{Constants}}: for every $q \in \{1,2\}$ and every real number $A > 0$, the constant random variable $A$ satisfies $A \leq \cO_q(2A)$.
\end{enumerate}
\end{lemma}

\begin{proof}
Items~(i) and~(v) follow from the definitions, and item~(ii) from the inequality $r \leq e^{r^2}$ for $r \geq 0$.
For item~(iii), the function $r \mapsto \exp(r^q)$ is convex on $[0,\infty)$, so Jensen's inequality bounds the exponential moment of the weighted average of the random variables $X_k/A_k$, with weights proportional to $A_k$, by the weighted average of their exponential moments.
The same argument, applied to the average of $X(s)/A$ over $s \in [0,T]$, gives item~(iv).
Items~(iii) and~(iv) are also special cases of \citet*[Lemma~A.4]{AKM2019}.
\end{proof}

\section{Homogenization of the heat equation} \label{sec:proof}

In this section we prove Theorem~\ref{thm:main}.
Throughout the section, we fix $f \in C^\infty_c(\R^d)$, we let $\bar u$ solve~\eqref{eq:heat.continuum}, and we use the notation $\bar u_s$ of Section~\ref{sec:prelim}.

\subsection{The deterministic term} \label{sec:mean}

The operator $\Lbar$ is
\begin{equation} \label{eq:def.Lbar}
    \Lbar v(x) := \frac{\eps^d}{\kappa_\eps} \sum_{z \in \eps \Z^d \setminus \{0\}} J(z) \bigl( v(x + z) - v(x) \bigr) \qquad \text{for } x \in \eps \Z^d\,.
\end{equation}
The next lemma shows that the continuum solution solves the lattice heat equation for $\Lbar$, up to an error of order $\kappa_\eps^{-1}$.

\begin{lemma}[Consistency with the continuum equation] \label{lem:mean}
There exists a constant $C = C(d, f) < \infty$ such that, for every $\eps \in (0, \nf12)$ and every $s \geq 0$,
\begin{equation*}
    \bigl\| \partial_s \bar u_s - \Lbar \bar u_s \bigr\|_{L^2(\eps \Z^d)} \leq \frac{C}{\kappa_\eps}\,.
\end{equation*}
\end{lemma}

\begin{proof}
We split the sum~\eqref{eq:def.Lbar} into the short jumps and the long jumps.
For the short jumps, we expand $\bar u_s(\cdot + z) - \bar u_s$ with~\eqref{eq:taylor}.
The linear terms cancel, since the set of short jumps is symmetric under $z \mapsto -z$.
By the identity~\eqref{eq:second.moment}, the quadratic terms give
\begin{equation*}
    \frac{\eps^d}{\kappa_\eps} \sum_{\substack{z \in \eps \Z^d \\ 0 < |z| \leq 1}} J(z) \frac{z \cdot \nabla^2 \bar u_s z}{2} = \frac{1}{2d} \Delta \bar u_s = \partial_s \bar u_s\,.
\end{equation*}
The remainder in~\eqref{eq:taylor} and the bound~\eqref{eq:increment} for the long jumps then give
\begin{equation*}
    \bigl\| \partial_s \bar u_s - \Lbar \bar u_s \bigr\|_{L^2(\eps \Z^d)} \leq \frac{C \eps^d}{\kappa_\eps} \Biggl( \sum_{\substack{z \in \eps \Z^d \\ 0 < |z| \leq 1}} \frac{1}{|z|^{d-1}} + \sum_{\substack{z \in \eps \Z^d \\ |z| > 1}} \frac{1}{|z|^{d+2}} \Biggr)\,.
\end{equation*}
A decomposition into dyadic scales shows that both sums, multiplied by $\eps^d$, are bounded by a constant $C(d)$.
\end{proof}

\subsection{The random term} \label{sec:forcing}

We now turn to the random term $(\Lbar - \Leps) \bar u_s$.

For $x, z \in \eps \Z^d$ with $z \neq 0$, the \textbf{canonical path} $\gamma_{x,z}$ is the nearest-neighbor path from $x$ to $x + z$ that first performs $|z_1|/\eps$ steps in direction $\sgn(z_1) e_1$, then $|z_2|/\eps$ steps in direction $\sgn(z_2) e_2$, and so on until the $d$-th coordinate (see Figure~\ref{fig:path}).
We represent the edge $\{y, y + \eps e_i\}$ by the pair $(y,i)$ with $y \in \eps \Z^d$ and $1 \leq i \leq d$, and we write $(y, i) \in \gamma_{x,z}$ to mean that the path $\gamma_{x,z}$ crosses this edge.
By construction, the path $\gamma_{x,z}$ has $|z|_1/\eps$ edges, and it is the translate by $x$ of the path $\gamma_{0,z}$.
It crosses its edges in direction $i$ with orientation $\sgn(z_i)$, so every function $v \colon \eps \Z^d \to \R$ satisfies the telescoping identity
\begin{equation} \label{eq:telescoping}
    v(x + z) - v(x) = \eps \sum_{(y,i) \in \gamma_{x,z}} \sgn(z_i) D_i v(y)\,.
\end{equation}

\begin{figure}[ht]
\centering
\begin{tikzpicture}[scale=0.7]
    \foreach \a in {0,...,6} \foreach \b in {0,...,4} \fill[black!40] (\a,\b) circle (1.6pt);
    \draw[very thick, blue, ->, >=stealth] (1,1) -- (5,1) -- (5,4);
    \draw[line width=2.4pt, red] (3,1) -- (4,1);
    \fill (1,1) circle (2.4pt) node[below=2pt] {$x$};
    \fill (5,4) circle (2.4pt) node[above left] {$x + z$};
    \fill[red] (3,1) circle (2.4pt) node[below=2pt] {$y$};
    \node[red] at (3.5,1.45) {$(y,1)$};
    \node[blue] at (5.6,2.5) {$\gamma_{x,z}$};
\end{tikzpicture}
\caption{A canonical path in dimension $d = 2$. The jump from $x$ to $x + z$ with $z = (4 \eps, 3 \eps)$ is routed along the path $\gamma_{x,z}$ (blue), which performs $4$ steps in direction $e_1$ and then $3$ steps in direction $e_2$. Exactly $|z_1|/\eps = 4$ translates of this path cross the red edge $(y,1)$.}
\label{fig:path}
\end{figure}

For every time $s \geq 0$, every vertex $y \in \eps \Z^d$ and every direction $1 \leq i \leq d$, we define
\begin{equation} \label{eq:def.G}
    G_i(s, y) := \frac{\eps^{d+1}}{2} \sum_{\substack{z \in \eps \Z^d \\ 0 < |z| \leq 1}} J(z) \sgn(z_i) \sum_{\substack{x \in \eps \Z^d \\ (y,i) \in \gamma_{x,z}}} \bigl( a_{x,z} - 1 \bigr) \bigl( \bar u_s(x + z) - \bar u_s(x) \bigr)\,,
\end{equation}
and we write $G(s, \cdot) := (G_1(s, \cdot), \ldots, G_d(s, \cdot))$.
The sum~\eqref{eq:def.G} is finite: for every short jump $z$, exactly $|z_i|/\eps$ translates of the path $\gamma_{0,z}$ cross the edge $(y,i)$.

\begin{proposition}[Representation of the random term] \label{prop:forcing}
Let $G$ be the vector field~\eqref{eq:def.G}.
Then there exists a constant $C = C(d, \lambda, f) < \infty$ such that, for every $\eps \in (0, \nf12)$, every realization of the conductances, every $s \geq 0$ and every $v \in L^2(\eps \Z^d)$,
\begin{equation} \label{eq:forcing}
    \Bigl| \bigl\langle \bigl( \Lbar - \Leps \bigr) \bar u_s, v \bigr\rangle - \frac{1}{\kappa_\eps} \langle G(s, \cdot), D v \rangle \Bigr| \leq \frac{C}{\kappa_\eps} \|v\|_{L^2(\eps \Z^d)}\,.
\end{equation}
\end{proposition}

\begin{proof}
The identity~\eqref{eq:IBP}, applied once to the operator $\Leps$ and once with every conductance equal to $1$, gives
\begin{equation} \label{eq:forcing.sum}
     \bigl\langle \bigl( \Lbar - \Leps \bigr) \bar u_s, v \bigr\rangle = \frac{\eps^{2d}}{2 \kappa_\eps} \sum_{x \in \eps \Z^d} \sum_{z \in \eps \Z^d \setminus \{0\}} \bigl( a_{x,z} - 1 \bigr) J(z) \bigl( \bar u_s(x + z) - \bar u_s(x) \bigr) \bigl( v(x + z) - v(x) \bigr)\,.
\end{equation}

We first bound the contribution of the long jumps.
For every long jump $z$, the Cauchy--Schwarz inequality and the bound~\eqref{eq:increment} give
\begin{equation*}
    \eps^d \sum_{x \in \eps \Z^d} \bigl| \bar u_s(x + z) - \bar u_s(x) \bigr| \bigl| v(x + z) - v(x) \bigr| \leq \|\bar u_s(\cdot + z) - \bar u_s\|_{L^2(\eps \Z^d)} \|v(\cdot + z) - v\|_{L^2(\eps \Z^d)} \leq C \|v\|_{L^2(\eps \Z^d)}\,.
\end{equation*}
Since $|a_{x,z} - 1| \leq \lambda^{-1}$ and $\eps^d \sum_{z \in \eps \Z^d, |z| > 1} J(z) \leq C$, the contribution of the long jumps to the right side of~\eqref{eq:forcing.sum} is at most $C \kappa_\eps^{-1} \|v\|_{L^2(\eps \Z^d)}$ in absolute value.

We next compute the contribution of the short jumps.
Inserting the telescoping identity~\eqref{eq:telescoping} for $v$ and exchanging the order of summation, we see from~\eqref{eq:def.G} that this contribution equals $\kappa_\eps^{-1} \langle G(s, \cdot), D v \rangle$.
\end{proof}

\begin{lemma}[Gaussian tails of the field $G$] \label{lem:G}
Let $G$ be the vector field~\eqref{eq:def.G}.
Then there exists a constant $C = C(d, \lambda, f) < \infty$ such that, for every $\eps \in (0, \nf12)$ and every $s \geq 0$,
\begin{equation*}
    \|G(s, \cdot)\|_{L^2(\eps \Z^d)} \leq \cO_2(C)\,.
\end{equation*}
\end{lemma}

\begin{proof}
We fix $\eps \in (0, \nf12)$ and $s \geq 0$, and we first consider a vertex $y \in \eps \Z^d$ and a direction $1 \leq i \leq d$.
The random variable $G_i(s, y)$ is a finite linear combination of the independent, centered and bounded random variables $a_{x,z} - 1$, indexed by the unordered pairs $\{x, x+z\}$.
The coefficient of each of them is the sum of at most two terms of the form $\nf12 \eps^{d+1} J(z) \sgn(z_i) (\bar u_s(x+z) - \bar u_s(x))$, one for each of the two orientations $(x,z)$ and $(x+z,-z)$ of the pair.
The sum $\sigma_i^2(s,y)$ of the squares of the coefficients therefore satisfies
\begin{equation} \label{eq:variance.G}
    \sigma_i^2(s,y) \leq C \eps^{2d+2} \sum_{\substack{z \in \eps \Z^d \\ 0 < |z| \leq 1}} J^2(z) \sum_{\substack{x \in \eps \Z^d \\ (y,i) \in \gamma_{x,z}}} \bigl( \bar u_s(x + z) - \bar u_s(x) \bigr)^2\,.
\end{equation}
If $\sigma_i(s,y) = 0$, then $G_i(s,y) = 0$.
Otherwise, Hoeffding's inequality gives $|G_i(s,y)| \leq \cO_2(C \sigma_i(s,y))$, so $G_i^2(s,y) \leq \cO_1(C \sigma_i^2(s,y))$ by Lemma~\ref{lem:O}(i).
Summing over the edges with Lemma~\ref{lem:O}(iii), we obtain
\begin{equation*}
    \|G(s, \cdot)\|_{L^2(\eps \Z^d)}^2 = \eps^d \sum_{y \in \eps \Z^d} \sum_{i=1}^d G_i^2(s,y) \leq \cO_1 \biggl( C \eps^d \sum_{y \in \eps \Z^d} \sum_{i=1}^d \sigma_i^2(s,y) \biggr)\,.
\end{equation*}
By Lemma~\ref{lem:O}(i), it remains to bound the sum on the right side by a constant.
Every path $\gamma_{x,z}$ has $|z|_1/\eps \leq C |z|/\eps$ edges, so summing~\eqref{eq:variance.G} over $y \in \eps \Z^d$ and $1 \leq i \leq d$ and applying the bound~\eqref{eq:increment} give
\begin{align*}
    \eps^d \sum_{y \in \eps \Z^d} \sum_{i=1}^d \sigma_i^2(s,y)
    & \leq C \eps^{2d+1} \sum_{\substack{z \in \eps \Z^d \\ 0 < |z| \leq 1}} J^2(z) |z| \biggl( \eps^d \sum_{x \in \eps \Z^d} \bigl( \bar u_s(x + z) - \bar u_s(x) \bigr)^2 \biggr) \\
    & \leq C \eps^{2d+1} \sum_{z \in \eps \Z^d \setminus \{0\}} \frac{1}{|z|^{2d+1}}\,.
\end{align*}
The substitution $z = \eps w$ shows that the right side equals $C \sum_{w \in \Z^d \setminus \{0\}} |w|^{-2d-1}$, which is finite in every dimension.
\end{proof}

\subsection{The energy estimate and the proof of Theorem~\ref{thm:main}} \label{sec:energy}

We now combine Lemma~\ref{lem:mean}, Proposition~\ref{prop:forcing} and Lemma~\ref{lem:G} with an energy estimate.

\begin{proof}[Proof of Theorem~\ref{thm:main}]
The case $T = 0$ is immediate, so we fix $T > 0$ and $\eps \in (0, \nf12)$.

\smallskip

\emph{Step~1: The energy estimate.}
We show that, for every realization of the conductances,
\begin{equation} \label{eq:final.deterministic}
    \sup_{0 \leq s \leq T} \|u_\eps(s, \cdot) - \bar u_s\|_{L^2(\eps \Z^d)} \leq \frac{C T}{\kappa_\eps} + \frac{C}{\kappa_\eps^{\nf12}} \biggl( \int_0^T \|G(s, \cdot)\|_{L^2(\eps \Z^d)}^2 \,ds \biggr)^{\nf12}\,.
\end{equation}
Let $w_s := u_\eps(s, \cdot) - \bar u_s$ for $s \geq 0$.
Since $u_\eps(0, \cdot) = \bar u_0 = f$ on $\eps \Z^d$, we have $w_0 = 0$.
By the decomposition~\eqref{eq:splitting}, the function $w$ solves
\begin{equation*}
    \partial_s w_s - \Leps w_s = - \bigl( \partial_s \bar u_s - \Lbar \bar u_s \bigr) - \bigl( \Lbar - \Leps \bigr) \bar u_s\,.
\end{equation*}
We test this equation against $w_s$ and use~\eqref{eq:IBP} to obtain, for every $s > 0$,
\begin{equation*}
    \frac12 \frac{d}{ds} \|w_s\|_{L^2(\eps \Z^d)}^2 + B_\ba(w_s, w_s) = - \bigl\langle \partial_s \bar u_s - \Lbar \bar u_s, w_s \bigr\rangle - \bigl\langle \bigl( \Lbar - \Leps \bigr) \bar u_s, w_s \bigr\rangle\,.
\end{equation*}
By Lemma~\ref{lem:mean}, the first term on the right side is at most $C \kappa_\eps^{-1} \|w_s\|_{L^2(\eps \Z^d)}$ in absolute value.
For the second term, we apply~\eqref{eq:forcing} with $v = w_s$, the Cauchy--Schwarz inequality, the coercivity bound~\eqref{eq:coercivity} and Young's inequality:
\begin{align*}
    \Bigl| \bigl\langle \bigl( \Lbar - \Leps \bigr) \bar u_s, w_s \bigr\rangle \Bigr|
    & \leq \frac{C}{\kappa_\eps} \|w_s\|_{L^2(\eps \Z^d)} + \frac{1}{\kappa_\eps} \|G(s, \cdot)\|_{L^2(\eps \Z^d)} \|D w_s\|_{L^2(\eps \Z^d)} \\
    & \leq \frac{C}{\kappa_\eps} \|w_s\|_{L^2(\eps \Z^d)} + B_\ba(w_s, w_s) + \frac{C}{\kappa_\eps} \|G(s, \cdot)\|_{L^2(\eps \Z^d)}^2\,.
\end{align*}
Combining these bounds, we obtain, for every $s > 0$,
\begin{equation*}
    \frac12 \frac{d}{ds} \|w_s\|_{L^2(\eps \Z^d)}^2 \leq \frac{C}{\kappa_\eps} \|w_s\|_{L^2(\eps \Z^d)} + \frac{C}{\kappa_\eps} \|G(s, \cdot)\|_{L^2(\eps \Z^d)}^2\,.
\end{equation*}
Let $S_T$ be the supremum of $\|w_s\|_{L^2(\eps \Z^d)}$ over $0 \leq s \leq T$.
Integrating the previous display over $[0,t]$ and taking the supremum over $0 \leq t \leq T$ gives
\begin{equation*}
    S_T^2 \leq \frac{C T}{\kappa_\eps} S_T + \frac{C}{\kappa_\eps} \int_0^T \|G(s, \cdot)\|_{L^2(\eps \Z^d)}^2 \,ds\,.
\end{equation*}
For nonnegative numbers $A_1$ and $A_2$, the inequality $S^2 \leq A_1 S + A_2$ implies $S \leq A_1 + A_2^{\nf12}$, which gives~\eqref{eq:final.deterministic}.

\smallskip

\emph{Step~2: The Gaussian bound.}
We bound the right side of~\eqref{eq:final.deterministic} by ${\cO_2(C (1 + T) |\ln \eps|^{-\nf12})}$.
By Lemmas~\ref{lem:G} and~\ref{lem:O}(i), we have $\|G(s, \cdot)\|_{L^2(\eps \Z^d)}^2 \leq \cO_1(C)$ for every $s \in [0, T]$.
Lemma~\ref{lem:O}(iv) then gives
\begin{equation*}
    \int_0^T \|G(s, \cdot)\|_{L^2(\eps \Z^d)}^2 \,ds \leq \cO_1(C T)\,,
\end{equation*}
so the square root of this integral is at most $\cO_2(C T^{\nf12})$ by Lemma~\ref{lem:O}(i).
By Lemma~\ref{lem:O}(iii) and~(v), the right side of~\eqref{eq:final.deterministic} is therefore at most
\begin{equation*}
    \cO_2 \biggl( C \Bigl( \frac{T}{\kappa_\eps} + \frac{T^{\nf12}}{\kappa_\eps^{\nf12}} \Bigr) \biggr)\,.
\end{equation*}
Since $T^{\nf12} \leq 1 + T$ and $\kappa_\eps \geq \max(2d, c |\ln \eps|)$, the argument of $\cO_2$ in this display is at most $C (1 + T) |\ln \eps|^{-\nf12}$, which proves~\eqref{eq:main}.
\end{proof}

\section{The invariance principle} \label{sec:IP}

In this section we prove Theorem~\ref{thm:IP}.
Let $(\eps_n)_{n \geq 1}$ be a sequence in $(0, \nf12)$ converging to $0$.
By \citet[Theorems~16.10 and~16.11]{Kallenberg2002}, the laws of the walks $X^{\eps_n}$ under $\PP^{\mathrm{ann}}_0$ converge to the law of $d^{-\nf12} W$ in $D([0, \infty), \R^d)$ if the following two conditions hold.
\begin{enumerate}
    \item[\textup{(i)}] The finite-dimensional marginals of the walks $X^{\eps_n}$ under $\PP^{\mathrm{ann}}_0$ converge to those of $d^{-\nf12} W$.
    \item[\textup{(ii)}] For every uniformly bounded sequence $(\tau_n)_{n \geq 1}$ of stopping times of the walks $X^{\eps_n}$ and every sequence $(h_n)_{n \geq 1}$ of positive numbers converging to $0$, the differences $X^{\eps_n}_{\tau_n + h_n} - X^{\eps_n}_{\tau_n}$ converge to $0$ in probability under $\PP^{\mathrm{ann}}_0$.
\end{enumerate}
Condition~(ii) is Aldous' criterion~\citep{Aldous1978}.
We prove condition~(ii) in Proposition~\ref{prop:aldous} and condition~(i) in Proposition~\ref{prop:fdd}.
For fixed $\eps$, the walk $X^\eps$ has a bounded total jump rate, so it has the strong Markov property and $e^{t \Leps} v(x) = \E^\ba_x[v(X^\eps_t)]$ for every bounded function $v$.
Under the annealed law, the walk need not be a Markov process, so we argue under the quenched law and average over the conductances at the end.

\subsection{Heat-kernel bounds} \label{sec:hk}

The heat kernel of $X^\eps$ with respect to the counting measure multiplied by $\eps^d$ is
\begin{equation} \label{eq:def.hk}
    p^\eps_h(x, y) := \eps^{-d} \PP^\ba_x \left[ X^\eps_h = y \right] \qquad \text{for } h > 0 \text{ and } x, y \in \eps \Z^d\,.
\end{equation}
The operator $\Leps$ is self-adjoint with respect to this measure, so $p^\eps_h$ is symmetric in $x$ and $y$.

\begin{proposition}[Heat-kernel bounds] \label{prop:hk}
There exists a constant $C = C(d, \lambda) < \infty$ such that, for every $\eps \in (0, \nf12)$, every realization of the conductances and all $x, y \in \eps \Z^d$, the heat kernel~\eqref{eq:def.hk} satisfies the following bounds.
\begin{enumerate}
    \item[\textup{(i)}] \underline{\emph{On-diagonal bound}}: for every $h \geq \eps$,
    \begin{equation} \label{eq:hk.diag}
        p^\eps_h(x, y) \leq \frac{C}{h^{\nf{d}{2}}}\,.
    \end{equation}
    \item[\textup{(ii)}] \underline{\emph{Displacement bound}}: for every $h > 0$ and every $R \geq \eps$,
    \begin{equation} \label{eq:hk.displacement}
        \PP^\ba_x \left[ |X^\eps_h - x| > R \right] \leq \frac{C h \left( 1 + \ln (1 + R) \right)}{R^2}\,.
    \end{equation}
\end{enumerate}
\end{proposition}

\begin{proof}
Let $Y$ be the random walk on $\Z^d$ that jumps from $x$ to $x + z$ at rate $\ba(\{x, x+z\}) |z|^{-d-2}$, let $L$ be its generator, and let $p_t(x,y) := \PP^\ba_x[Y_t = y]$.
The process $(\eps Y_{h/(\kappa_\eps \eps^2)})_{h \geq 0}$ started from $x/\eps$ has the same law as $X^\eps$ started from $x$, so that
\begin{equation} \label{eq:hk.rescaling}
    p^\eps_h(x,y) = \eps^{-d} p_t \Bigl( \frac{x}{\eps}, \frac{y}{\eps} \Bigr) \qquad \text{with} \qquad t := \frac{h}{\kappa_\eps \eps^2}\,.
\end{equation}

\smallskip

\emph{Step~1: The on-diagonal bound.}
We show that, for every $t > 0$ and all $x, y \in \Z^d$,
\begin{equation} \label{eq:hk.micro.diag}
    p_t(x, y) \leq \frac{C}{\left(1 + t \ln (2 + t) \right)^{\nf{d}{2}}}\,,
\end{equation}
and we deduce~\eqref{eq:hk.diag}.
Let $\Lambda := \lambda^{-1} \sum_{z \in \Z^d \setminus \{0\}} |z|^{-d-2}$, which bounds the total jump rate of $Y$ at every vertex.
The operator $K := I + (2 \Lambda)^{-1} L$ is a Markov operator, symmetric with respect to the counting measure, whose kernel satisfies $K(x,x) \geq \nf12$ and $c |x-y|^{-d-2} \leq K(x,y) \leq C |x-y|^{-d-2}$ for $x \neq y$.
By \citet[Theorem~1.1]{MSC2015}, applied with escape time exponent $2$ and kernel scale $r \mapsto (1+r)^2$, we have $K^{2n}(x,x) \leq C (1 + n \ln(2+n))^{-\nf{d}{2}}$ for every integer $n \geq 0$ and every $x \in \Z^d$.
The constant $C$ of~\citet{MSC2015} depends on $K$ only through the constants in these bounds on $K$, hence only on $d$ and $\lambda$.
The Cauchy--Schwarz inequality $K^n(x,y) \leq (K^{2 \lfloor n/2 \rfloor}(x,x) K^{2 \lceil n/2 \rceil}(y,y))^{\nf12}$ extends this bound to $K^n(x,y)$.
Since $p_t(x,y)$ is the average of $K^n(x,y)$ over a Poisson distribution of $n$ with mean $2 \Lambda t$, the bound~\eqref{eq:hk.micro.diag} follows from the concentration of this distribution.
If $h \geq \eps$, then the time $t$ of~\eqref{eq:hk.rescaling} satisfies $t \geq 1/(\eps \kappa_\eps)$, so $\ln(2 + t) \geq c |\ln \eps| \geq c \kappa_\eps$, and~\eqref{eq:hk.rescaling} and~\eqref{eq:hk.micro.diag} give $p^\eps_h(x,y) \leq C \eps^{-d} (t \kappa_\eps)^{-\nf{d}{2}} = C h^{-\nf{d}{2}}$.

\smallskip

\emph{Step~2: The off-diagonal bound.}
We show that, for every $t > 0$ and all distinct $x, y \in \Z^d$,
\begin{equation} \label{eq:offdiagonal}
   p_t(x, y) \leq \frac{C t \ln(2 + |x - y|)}{|x - y|^{d+2}}\,.
\end{equation}
To prove this result, we rely on the general theory of \citet*{CKW2021Stability} which provides upper and lower bounds for a large class of heat kernels (and on the result of \citet{Malmquist2023}, which extends their results to functions defined on graphs). We will in particular use \citet*[Theorem~1.15, (5) implies (1)]{CKW2021Stability} and check the three conditions of (5), with $\phi$ as in~\eqref{eq:hk.scale}: the Faber--Krahn inequality FK$(\phi)$, the upper bound $J_{\phi, \leq}$ on the decay of the long-range interaction and the cutoff Sobolev inequality CSJ$(\phi)$.

We consider $\Z^d$ with the Euclidean distance and the counting measure, and the Dirichlet form of the walk $Y$, whose jump kernel is $\ba(\{x,y\}) |x-y|^{-d-2}$ up to the factor $\nf12$ that comes from the normalization of the Dirichlet form in these references.
For $r > 0$, let $V(r)$ denote the number of points of $\Z^d$ in an open ball of radius $r$ centered at a point of $\Z^d$, which equals $1$ for $r \leq 1$ and is comparable to $r^d$ for $r > 1$.
We use the scale function
\begin{equation} \label{eq:hk.scale}
    \phi(r) := \frac{r^2 \ln 3}{\ln(2 + r)} \qquad \text{for } r \geq 0\,.
\end{equation}
This function is continuous and increasing, it satisfies $\phi(0) = 0$ and $\phi(1) = 1$, and the ratio $\phi(r_1)/\phi(r_2)$ lies between $(r_1/r_2)^{\nf32}$ and $(r_1/r_2)^2$ for all $r_1 \geq r_2 > 0$.
The following three conditions hold on $\Z^d$.
\begin{itemize}
    \item \emph{Faber--Krahn inequality FK$(\phi)$.} By \citet*[Proposition~7.6]{CKW2021Stability}, it is enough to prove the on-diagonal bound $p_t(x,x) \leq C / V(\phi^{-1}(t))$ for every $t > 0$ and every $x \in \Z^d$. This is implied by~\eqref{eq:hk.micro.diag} and the inequality $V(\phi^{-1}(t)) \leq C (1 + t \ln(2+t))^{\nf{d}{2}}$ for every $t > 0$.
    \item \emph{Upper bound $J_{\phi,\le}$ on the jump kernel.} The bound $\ba(\{x,y\}) \leq \lambda^{-1}$ gives $\ba(\{x,y\}) |x-y|^{-d-2} \leq C / (V(|x-y|) \phi(|x-y|))$ for all distinct $x, y \in \Z^d$.
    \item \emph{Cutoff Sobolev inequality $\mathrm{CSJ}(\phi)$.} For $x_0 \in \Z^d$ and $r_1 \geq r_2 > 0$, the cutoff function $\chi(x) := \max(0, 1 - \max(0, |x - x_0| - r_1) / r_2)$ has finite support and satisfies $|\chi(x) - \chi(y)| \leq \min(1, |x-y|/r_2)$, so, for every $x \in \Z^d$,
    \begin{equation*}
        \sum_{y \in \Z^d \setminus \{x\}} \frac{\ba(\{x,y\})}{|x - y|^{d+2}} \bigl( \chi(x) - \chi(y) \bigr)^2
        \leq C \sum_{z \in \Z^d \setminus \{0\}} \min \Bigl( 1, \frac{|z|^2}{r_2^2} \Bigr) \frac{1}{|z|^{d+2}}
        \leq \frac{C}{\phi(r_2)}\,.
    \end{equation*}
    For $r_2 \geq 1$, the sum over $|z| \leq r_2$ is at most $C r_2^{-2} \ln(2+r_2)$, and the remaining sum is at most $C r_2^{-2}$; for $r_2 < 1$, the full sum is bounded by $C$.
    Multiplying this pointwise bound by the square of a function in the form domain and summing over the ball in the definition of $\mathrm{CSJ}(\phi)$ gives this condition.
\end{itemize}
These conditions imply, for every $t > 0$ and all distinct $x, y \in \Z^d$,
\begin{equation*}
    p_t(x,y) \leq C \min \Bigl( \frac{1}{V(\phi^{-1}(t))}, \frac{t}{V(|x-y|) \phi(|x-y|)} \Bigr)\,,
\end{equation*}
which gives~\eqref{eq:offdiagonal}.
The constants in these results depend only on the constants in their hypotheses, so, as in Step~1, the constant $C$ depends only on $d$ and $\lambda$.

\smallskip

\emph{Step~3: The displacement bound.}
We deduce~\eqref{eq:hk.displacement} from~\eqref{eq:offdiagonal}.
Summing~\eqref{eq:offdiagonal} over the vertices $y \in \Z^d$ with $|y - x/\eps| > R/\eps$ and using~\eqref{eq:hk.rescaling}, we obtain, for every $h > 0$ and every $R \geq \eps$,
\begin{equation*}
    \PP^\ba_x \left[ |X^\eps_h - x| > R \right] \leq \frac{C h}{\kappa_\eps \eps^2} \sum_{\substack{w \in \Z^d \\ |w| > R/\eps}} \frac{\ln(2 + |w|)}{|w|^{d+2}} \leq \frac{C h \ln(2 + R/\eps)}{\kappa_\eps R^2}\,.
\end{equation*}
The bounds $\ln(2 + R/\eps) \leq |\ln \eps| + \ln(2 + R)$ and $\kappa_\eps \geq c |\ln \eps|$ give~\eqref{eq:hk.displacement}.
\end{proof}

\subsection{Aldous' criterion} \label{sec:tightness}

We verify condition~(ii) with the displacement bound~\eqref{eq:hk.displacement} and the strong Markov property.

\begin{proposition}[Aldous' criterion] \label{prop:aldous}
Let $(\eps_n)_{n \geq 1}$ be a sequence in $(0, \nf12)$ converging to $0$, let $M > 0$, and, for every $n \geq 1$, let $\tau_n \leq M$ be a stopping time of the walk $X^{\eps_n}$ started from $0$.
Let $(h_n)_{n \geq 1}$ be a sequence of positive numbers converging to $0$.
Then, under $\PP^{\mathrm{ann}}_0$, the differences $X^{\eps_n}_{\tau_n + h_n} - X^{\eps_n}_{\tau_n}$ converge to $0$ in probability as $n \to \infty$.
\end{proposition}

\begin{proof}
We fix $\eta > 0$.
A stopping time of the walk $X^{\eps_n}$ is a stopping time under $\PP^\ba_0$ for every realization of the conductances.
For every $n$ with $\eps_n \leq \eta$, the strong Markov property at time $\tau_n$ and the displacement bound~\eqref{eq:hk.displacement} with $R = \eta$ give
\begin{equation*}
    \PP^\ba_0 \left[ |X^{\eps_n}_{\tau_n + h_n} - X^{\eps_n}_{\tau_n}| > \eta \right] = \E^\ba_0 \left[ \PP^\ba_{X^{\eps_n}_{\tau_n}} \left[ |X^{\eps_n}_{h_n} - X^{\eps_n}_0| > \eta \right] \right] \leq \frac{C h_n (1 + \ln(1 + \eta))}{\eta^2}\,.
\end{equation*}
The right side does not depend on the conductances, and it tends to $0$ as $n \to \infty$.
Taking the expectation over the conductances gives the same bound under $\PP^{\mathrm{ann}}_0$.
\end{proof}

\subsection{Convergence of the finite-dimensional marginals} \label{sec:fdd}

We denote by $\mathbf{E}$ the expectation with respect to the Brownian motion $W$.

\begin{proposition}[Finite-dimensional marginals] \label{prop:fdd}
Let $W$ be a standard Brownian motion in $\R^d$, let $k \geq 1$ be an integer, and let $0 \leq t_1 < \cdots < t_k$.
Then, as $\eps \to 0$, the law of $(X^{\eps}_{t_1}, \ldots, X^{\eps}_{t_k})$ under $\PP^{\mathrm{ann}}_0$ converges to the law of $d^{-\nf12} (W_{t_1}, \ldots, W_{t_k})$.
\end{proposition}

\begin{proof}
Let $\bar P_s := e^{s \Delta / (2d)}$ be the heat semigroup of~\eqref{eq:heat.continuum}.
By Theorem~\ref{thm:main} and Lemma~\ref{lem:O}(ii), for every $F \in C^\infty_c(\R^d)$, there exists a constant $C = C(d, \lambda, F) < \infty$ such that, for every $s \geq 0$ and every $\eps \in (0, \nf12)$,
\begin{equation} \label{eq:thm.expectation}
    \E \left[ \| e^{s \Leps} F - \bar P_s F \|_{L^2(\eps \Z^d)} \right] \leq \frac{C (1 + s)}{|\ln \eps|^{\nf12}}\,.
\end{equation}

\smallskip

\emph{Step~1: One time.}
We fix $t > 0$ and $g \in C^\infty_c(\R^d)$, and we show that
\begin{equation} \label{eq:fdd.k1}
    \E^{\mathrm{ann}}_0 \left[ g(X^\eps_t) \right] \to \mathbf{E} \left[ g(d^{-\nf12} W_t) \right] = \bar P_t g(0) \qquad \text{as } \eps \to 0\,.
\end{equation}
We fix a time $h \in (0, \min(t,1))$, which we send to $0$ at the end, and we let $g_h := \bar P_{t - h} g$.
The heat semigroup commutes with derivatives and contracts the supremum norm, so $\|g_h\|_{L^\infty(\R^d)} \leq \|g\|_{L^\infty(\R^d)}$ and $\|\nabla g_h\|_{L^\infty(\R^d)} \leq \|\nabla g\|_{L^\infty(\R^d)}$.
The Markov property at time $h$ gives
\begin{equation*}
    \E^\ba_0 \left[ g(X^\eps_t) \right] = \eps^d \sum_{y \in \eps \Z^d} p^\eps_h(0, y) e^{(t - h) \Leps} g (y)
    \quad \text{and} \quad
    \E^\ba_0 \left[ g_h(X^\eps_h) \right] = \eps^d \sum_{y \in \eps \Z^d} p^\eps_h(0, y) g_h(y)\,.
\end{equation*}
The symmetry of the heat kernel and the semigroup property give $\eps^d \sum_{y \in \eps \Z^d} \bigl( p^\eps_h(0, y) \bigr)^2 = p^\eps_{2h}(0, 0)$.
For every $\eps \leq 2h$, the Cauchy--Schwarz inequality and the on-diagonal bound~\eqref{eq:hk.diag} therefore give
\begin{equation*}
    \left| \E^\ba_0 \left[ g(X^\eps_t) \right] - \E^\ba_0 \left[ g_h(X^\eps_h) \right] \right| \leq \frac{C}{h^{\nf{d}{4}}} \| e^{(t - h) \Leps} g - \bar P_{t - h} g \|_{L^2(\eps \Z^d)}\,.
\end{equation*}
For fixed $h$, the expectation of the right side tends to $0$ as $\eps \to 0$ by~\eqref{eq:thm.expectation}.
Next, we split according to whether $|X^\eps_h| \leq h^{\nf13}$.
The bound $|g_h(y) - g_h(0)| \leq \|\nabla g\|_{L^\infty(\R^d)} |y|$ and the displacement bound~\eqref{eq:hk.displacement} with $R = h^{\nf13}$ give, for every realization of the conductances and every $\eps \leq h^{\nf13}$,
\begin{equation*}
    \left| \E^\ba_0 \left[ g_h(X^\eps_h) \right] - g_h(0) \right| \leq \|\nabla g\|_{L^\infty(\R^d)} h^{\nf13} + C \|g\|_{L^\infty(\R^d)} h^{\nf13}\,.
\end{equation*}
Finally, the identity $\partial_s \bar P_s g = \nf{1}{2d} \bar P_s \Delta g$ gives $|g_h(0) - \bar P_t g(0)| \leq h \|\Delta g\|_{L^\infty(\R^d)}$.
Combining the three previous bounds and taking the expectation over the conductances, we obtain a constant $C = C(d, \lambda, g) < \infty$ such that $\limsup_{\eps \to 0} | \E^{\mathrm{ann}}_0 [ g(X^\eps_t) ] - \bar P_t g(0) | \leq C h^{\nf13}$.
Letting $h \to 0$ proves~\eqref{eq:fdd.k1}.

\smallskip

\emph{Step~2: Several times.}
Since $X^\eps_0 = 0 = W_0$, we may assume that $t_1 > 0$.
By the Stone--Weierstrass theorem, the linear combinations of the products $g_1 \otimes \cdots \otimes g_k$ of functions $g_1, \ldots, g_k \in C^\infty_c(\R^d)$ are dense in the space $C_0((\R^d)^k)$ of continuous functions vanishing at infinity, so it suffices to fix $g_1, \ldots, g_k \in C^\infty_c(\R^d)$ and show that
\begin{equation} \label{eq:fdd.k}
    \E^{\mathrm{ann}}_0 \biggl[ \prod_{j=1}^k g_j(X^\eps_{t_j}) \biggr] \to \mathbf{E} \biggl[ \prod_{j=1}^k g_j(d^{-\nf12} W_{t_j}) \biggr] \qquad \text{as } \eps \to 0\,.
\end{equation}
We let $F^\eps_k := g_k$ and define inductively
\begin{equation*}
    F^\eps_j := g_j e^{(t_{j+1} - t_j) \Leps} F^\eps_{j+1} \qquad \text{for } 1 \leq j \leq k-1\,.
\end{equation*}
We define $\bar F_j \in C^\infty_c(\R^d)$ in the same way with the semigroup $\bar P$.
The Markov property gives, for every realization of the conductances,
\begin{equation*}
    \E^\ba_0 \biggl[ \prod_{j=1}^k g_j(X^\eps_{t_j}) \biggr] = e^{t_1 \Leps} F^\eps_1(0)\,,
\end{equation*}
and the right side of~\eqref{eq:fdd.k} equals $\bar P_{t_1} \bar F_1(0)$.
The semigroup $e^{s \Leps}$ contracts $L^2(\eps \Z^d)$, so~\eqref{eq:thm.expectation} and a backward induction on $j$ show that, for every $1 \leq j \leq k$,
\begin{equation*}
    \E \left[ \|F^\eps_j - \bar F_j\|_{L^2(\eps \Z^d)} \right] \to 0 \qquad \text{as } \eps \to 0\,.
\end{equation*}
As in Step~1, the on-diagonal bound~\eqref{eq:hk.diag} shows that replacing $F^\eps_1$ by $\bar F_1$ in $e^{t_1\Leps}F^\eps_1(0)$ changes its annealed expectation by $o(1)$.
Step~1, applied to $\bar F_1$, then proves~\eqref{eq:fdd.k}.
\end{proof}

\end{document}